\documentclass[12pt]{amsart}

\usepackage{amsfonts,amssymb,amsmath}
\usepackage{graphicx}
\usepackage{psfrag}
\usepackage{color,xcolor}
\usepackage{verbatim}
\usepackage{epstopdf}
\usepackage[text={6in,9in},centering]{geometry}

\newtheorem{theo}{Theorem}[section]
\newtheorem{lem}[theo]{Lemma}
\newtheorem{prop}[theo]{Proposition}
\newtheorem{cor}[theo]{Corollary}

\theoremstyle{definition}

\theoremstyle{remark}
\newtheorem{rem}[theo]{Remark}

\newtheorem{exam}{Example}[section]
\newtheorem{conj}{Conjecture}[section]

\numberwithin{equation}{section}

\newcommand{\bR}{{\mathbb R}}
\newcommand{\bZ}{{\mathbb Z}}
\newcommand{\bZp}{{{\mathbb Z}^+}}

\newcommand{\bfS}{{\mathbf S}}

\newcommand{\red}{\textcolor{red}}
\newcommand{\blue}{\textcolor{blue}}

\def\wdt{\widetilde}

\newcommand{\upM}{{\overline{M}}}
\newcommand{\lowM}{{\underline{M}}}

\newcommand{\Var}{\mathrm{Var}}

\def\ga{\alpha}

\def\gl{\lambda}

\def\upBdim{\overline{\dim}_B}
\def\lowBdim{\underline{\dim}_B}
\def\Bdim{\dim_B}

\newcommand{\card}{{\rm card\,}}

\begin{document}

\title[Box dimension of generalized affine FIFs (II)]{Box dimension of generalized affine \\ fractal interpolation functions (II)}

\author{Lai Jiang}
\address{School of Mathematical Sciences, Zhejiang University, Hangzhou 310058, China}
\email{jianglai@zju.edu.cn}

\author{Huo-Jun Ruan}
\address{School of Mathematical Sciences, Zhejiang University, Hangzhou 310058, China}
\email{ruanhj@zju.edu.cn}
\thanks{The research was supported in part by NSFC grant 12371089, ZJNSF grant LY22A010023, and the Fundamental Research Funds for the Central Universities of China grant 226-2024-00136.}

\thanks{Corresponding author: Huo-Jun Ruan}


\date{}


\begin{abstract}
  Let $f$ be a generalized affine fractal interpolation function with vertical scaling functions. In this paper, we prove the monotonicity of spectral radii of vertical scaling matrices without additional assumptions. We also obtain the irreducibility of these matrices under certain conditions. By these results, we estimate $\Bdim \Gamma f$, the box dimension of the graph of $f$, by the limits of spectral radii of vertical scaling matrices. We also estimate $\Bdim \Gamma f$ directly by the sum function of vertical scaling functions. As an application, we study the box dimension of the graph of a generalized Weierstrass-type function. 
\end{abstract}

\maketitle

\section{Introduction}
\label{intro}

Fractal interpolation functions (FIFs) were introduced by Barnsley \cite{Bar86} in 1986. Basically, an FIF $f$ is a function which interpolates given data and its graph is the invariant set of an iterated function system (IFS). 

The generalized affine FIFs are an important class of FIFs. There are many works on these FIFs, including theoretical analysis \cite{All20,BRS20,BH89,CKV15,Dubuc18,WY13} and applications \cite{MaHa92,VDL94}. In particular, there are some works on the box dimension of the graphs of generalized affine FIFs \cite{BaMa15,Bed89-ConAppr,Bed89-Nonlinearity,Fen08,JhaVer21,RSY09}. In \cite{JiangRuan}, the authors introduced  vertical scaling matrices. We also obtained the monotonicity of spectral radii and irreducibility of these matrices under certain conditions. Then we estimated the box dimension of generalized affine FIFs by the limits of the spectral radii of these matrices. 

In the present paper, we continue the study of vertical scaling matrices and the box dimension of generalized affine FIFs. Mainly, we prove the monotonicity of spectral radii of vertical scaling matrices without additional assumptions, and obtain the irreducibility of vertical scaling matrices under weaker conditions than that in \cite{JiangRuan}. By using these results, we estimate the box dimension by  the limits of vertical scaling matrices under weaker conditions. We also estimate the box dimension directly by the sum function of vertical scaling functions. We remark that the class of generalized affine FIFs in the present paper is more general than the setting in \cite{JiangRuan}, so that our results are applicable to some classical fractal functions, including the classical Weierstrass functions.

The paper is organized as follows. In section~2, we recall some basic definitions and present main results.  In section 3, we study the monotonicity of spectral radii and irreducibility of vertical scaling matrices. By using these results, in section~4, we estimate the box dimension of generalized  affine FIFs by the limits of radii of vertical scaling matrices under certain conditions. In section~5, we estimate the box dimension of generalized affine FIFs by the sum function of the vertical scaling functions. In section~6, we apply our results to study the box dimension of a generalized Weierstrass-type function, and make some further remarks.

\section{Preliminaries and main results}
\subsection{The definition of generalized affine FIFs}

Let $N\geq 2$ be a positive integer. Given a data set $\{(x_n,y_n)\}_{n=0}^N\subset \bR^2$ with $x_0<x_1<\ldots<x_N$,  we define a family of functions $\{W_n\}_{n=1}^N$ from $[x_0,x_N]\times \bR$ to $[x_0,x_N]\times \bR$  by
\[
  W_n(x,y)=(a_n x+b_n, S_n(x) y+q_n(x)), \quad 1\leq n\leq N,
\]
such that for each $n$, $a_n$ and $b_n$ are real numbers, $S_n$ and $q_n$ are continuous functions on $[x_0,x_N]$ with $|S_n(x)|<1$ for all $x\in [x_0,x_N]$ and
\[
	W_n(x_0,y_0)=(x_{n-1},y_{n-1}), \quad W_n(x_N,y_N)=(x_{n},y_{n}).
\]
According to Barnsley's classical result \cite{Bar86} , there exists a unique continuous function $f$ on $[x_0,x_N]$ such that its graph $\Gamma f:=\{(x,f(x)):\, x\in [x_0,x_N]\}$ is the invariant set of the iterated function system (IFS for short) $\{W_n:\, 1\leq n\leq  N\}$, i.e.,
\begin{equation}\label{eq:1-5}
	\Gamma f=\bigcup_{n=1}^N W_n(\Gamma f).
\end{equation}
Furthermore, the function $f$ always interpolates the data set, i.e., $f(x_n)=y_n$ for all $0\leq n\leq N$.
The function $f$ is called the \emph{generalized affine fractal interpolation function} (generalized affine FIF for short) determined by the IFS $\{W_n\}_{n=1}^N$.

In the present paper, we will study $\dim_B \Gamma f$, the box dimension of the graph of $f$,  
where the following conditions are satisfied for each $n$:
\begin{itemize}
	\item[(A1)] $x_n-x_{n-1}=(x_N-x_0)/N$,
	\item[(A2)] $S_n$ is of bounded variation on $[x_0,x_N]$ and $|S_n(x)|<1$ for all $x\in [x_0,x_N]$,
	\item[(A3)] $q_n$ is of bounded variation on $[x_0,x_N]$.
\end{itemize}

\subsection{Main results}

In the rest of the paper, we write $I=[x_0,x_N]$ for simplicity.
We define a function $\gamma$ on $I$ by
\begin{equation*}
	\gamma(x)=\sum_{n=1}^{N} |S_n(x)|.
\end{equation*}
We call $\gamma$ the \emph{sum function} of the family of vertical scaling functions $\bfS=\{S_n\}_{n=1}^N$. Write $\gamma^*=\max_{x \in I} \gamma(x)$ and $\gamma_*=\min_{x \in I} \gamma(x)$.

Given a closed interval $J=[a,b]$, for each $k\in \mathbb{Z}^+$ and $1\leq j\leq N^k$, we write
\begin{equation}\label{def:J-jk}
	J_j^k=\Big[a+\frac{j-1}{N^k}(b-a),a+\frac{j}{N^k}(b-a)\Big].
\end{equation}

Given $k\in \bZp$, $1\leq n\leq N$ and $1\leq j\leq N^k$, we set
\[
  \overline{s}_{n,j}^k=\max_{x \in I_{j}^k} |S_n(x)|, \qquad \underline{s}_{n,j}^k=\min_{x \in I_{j}^k} |S_n(x)|.
\]

Now, for every $k\in \bZp$, we define a matrix $\upM_{k}$ by setting for $1\leq n\leq N$, $1\leq \ell\leq N^{k-1}$ and $1\leq j\leq N^k$,
\begin{equation*}
	(\upM_k)_{(n-1)N^{k-1}+\ell,j}=\begin{cases}
		\overline{s}^k_{n,j}, & \mbox{if } (\ell-1)N< j\leq \ell N, \\
		0, & \mbox{otherwise}.
	\end{cases}
\end{equation*}
Similarly, we define another $N^k \times N^k$ matrix $\lowM_k$ by
replacing $\overline{s}_{n,j}^k$ with $\underline{s}_{n,j}^k$. We call $\upM_k$ (resp. $\lowM_k$) the upper (resp. the lower) \emph{vertical scaling matrix} with level-$k$.

In this paper, we prove the monotonicity of spectral radii of vertical scaling matrices without additional assumptions. We also obtain the irreducibility of lower vertical scaling matrices under a weaker condition than that in \cite{JiangRuan}.


\begin{theo}\label{thm:Main-Results-1}
	With previous notations, we have
	\begin{enumerate}
		\item $\rho(\upM_{k})$, the spectral radius of $\upM_{k}$, is decreasing with respect to $k$. As a result, $\rho^*=\lim_{k \to \infty} \rho(\upM_{k})$ exists.
		\item $\rho(\lowM_{k})$ is increasing with respect to $k$. As a result, $\rho_*=\lim_{k \to \infty} \rho(\lowM_{k})$ exists.
		\item If $|S_n|$ is positive on $I$ for all $1\leq n\leq N$, then $\rho_*=\rho^*$.
		\item If $\gamma_*\geq 1$ and $S_n$ has only finitely many zero points on $I$ for all $1\leq n\leq N$, then $\lowM_{k}$ is primitive for sufficiently large enough $k\in \bZp$.
	\end{enumerate}
\end{theo}

By using same arguments in \cite{JiangRuan}, we can show that $\upM_{k}$ is primitive for all $k\in \bZp$ if the function $S_n$ is not identically zero on every subinterval of $I$ for all $1 \leq n \leq N$.

Let $\Var(f,I)$ be the classical total variation of $f$ on $I$.
By using Theorem~\ref{thm:Main-Results-1} and the irreduciblity of $\upM_{k}$, we can obtain the estimate of the box dimension of $\Gamma f$.
\begin{theo}\label{thm:Main-Results-2}
	Let $f$ be a generalized affine FIF satisfying conditions (A1)-(A3). Then we have the following results on the box dimension of $\Gamma f$.
	\begin{enumerate}
		\item Assume that the function $S_n$ is not identically zero on every subinterval of $I$ for all $1 \leq n \leq N$.
		Then
		$
		\overline{\dim}_{B} \Gamma f  \leq  \max\{1,1+ \log_N  \rho^*\}.
		$
		
		\item Assume that $\gamma_* \geq 1$ and the function $S_n$ has only finitely many zero points on $I$ for all $1 \leq n \leq N$. If $\Var(f,I)=\infty$, then
		$\lowBdim \Gamma f		\geq 1+ \log_N \rho_*$. 
		\item Under the assumption of the previous item and the additional assumption that $\rho_*=\rho^*$, if $\mathrm{Var}(f,I)=\infty$
		and $\rho_{\bfS}>1$, then
		\[
		\Bdim \Gamma f=1+\log_N  \rho_{\bfS}, 
		\]
		otherwise $\Bdim  \Gamma f=1$. Here $\rho_{\bfS}$ is the common value of $\rho_*$ and $\rho^*$.
	\end{enumerate}
\end{theo}

We remark that the assumption $\gamma_*\geq 1$ in Theorems~\ref{thm:Main-Results-1} and \ref{thm:Main-Results-2} can be replaced by a weaker assumption $\mathcal{ZM}(\mathbf{S})\geq N-2$. Please see section~3 for the definition of $\mathcal{ZM}(\mathbf{S})$.

It is easy to see that 
\[
  \gamma_*\leq \rho_* \leq \rho^*\leq \gamma^*.
\]
In section~5, we estimate the box dimension of $\Gamma f$ by $\gamma_*$ and $\gamma^*$ under weaker assumptions than these in Theorem~\ref{thm:Main-Results-2}.  
Akhtar, Prasad and Navascu\'{e}s \cite{APN17} used $S_{\max}=\max\{|S_i(x)|:\, x\in [0,1], 1\leq i\leq N\}$ and $S_{\min}=\min\{|S_i(x)|:\, x\in [0,1], 1\leq i\leq N\}\}$ to estimate the box dimension of $\alpha$-fractal functions, which is a special class of generalized affine FIFs. Our results are better than that in \cite{APN17}.

\section{Analysis on vertical scaling matrices}

\subsection{Some well known theorems and definitions}

We recall some notations and definitions in matrix analysis \cite{HorJoh90}.
Given a matrix $A=(a_{ij})_{n\times n}$, we say $A$ is \emph{nonnegative} (resp. \emph{positive}), denoted by $A\geq 0$ (resp. $A>0$), if $a_{ij}\geq0$ (resp. $a_{ij}>0$) for all $i$ and $j$. Let $B=(b_{ij})_{n\times n}$ be another matrix. We write $A\geq B$ (resp. $A>B$) if $a_{ij}\geq b_{ij}$ (resp. $a_{ij}>b_{ij}$) for all $i$ and $j$. Similarly, given $u=(u_1,\ldots,u_n),v=(v_1,\ldots,v_n)\in \bR^n$, we write $u\geq v$ (resp. $u>v$) if $u_i\geq v_i$ (resp. $u_i>v_i$) for all $i$.


A nonnegative matrix $A=(a_{ij})_{n\times n}$ is called \emph{irreducible} if for any $i,j\in \{1,\ldots,n\}$, there exists a finite sequence $i_0,\ldots,i_t\in \{1,\ldots,n\}$ such that $i_0=i,i_t=j$ and $a_{i_{\ell-1},i_\ell}>0$ for all $1\leq \ell \leq t$. $A$ is called \emph{primitive} if there exists $k\in \bZ^+$ such that $A^k>0$. It is clear that a primitive matrix is irreducible.

Given an $n\times n$ matrix $A$, we write $\sigma(A)$ the set of all eigenvalues of $A$ and define $\rho(A)=\max\{|\gl|:\, \gl\in \sigma(A)\}$. We call $\rho(A)$ the \emph{spectral radius} of $A$.

The following two lemmas are well known. Please see \cite[Chapter 8]{HorJoh90} for details.
\begin{lem}[Perron-Frobenius Theorem]\label{th:PF}
Let $A=(a_{ij})_{n\times n}$ be an irreducible nonnegative matrix. Then
\begin{enumerate}
	\item $\rho(A)$ is positive,
	\item $\rho(A)$ is an eigenvalue of $A$ and has a  positive eigenvector,
	\item $\rho(A)$ increases if any element of $A$ increases.
\end{enumerate}
\end{lem}

\begin{lem}\label{th:PFN}
Let $A=(a_{ij})_{n\times n}$ be a nonnegative matrix. Then $\rho(A)$ is an eigenvalue of $A$ and there is a nonnegative nonzero vector $x$ such that $Ax=\rho(A)x$.

\end{lem}

\subsection{Monotonicity of spectral radii of vertical scaling matrices}

\begin{theo}\label{thm:rho-upstar-exist}
  For all $k\in \bZ^+$,
  \[
    \rho(\upM_{k+1}) \leq  \rho (\upM_k).
  \]
  As a result, $\lim_{k\to \infty} \rho(\upM_k)$ exists, denoted by $\rho^*$.
\end{theo}
  In \cite{JiangRuan}, we proved this theorem under an additional assumption. Essentially, we required that $\upM_{k}$ are irreducible for all $k$. 
\begin{proof}
  Similarly as in \cite{JiangRuan}, we introduce another $N^{k+1}\times N^{k+1}$ matrix $\upM^*_k$ as follows:
\begin{equation*}
	(\upM_k^*)_{(n-1)N^{k}+\ell,j}=\begin{cases}
    \overline{s}^k_{n,\ell}, & \mbox{if } (\ell-1)N< j\leq \ell N, \\
    0, & \mbox{otherwise},
  \end{cases}
\end{equation*}
for $1\leq n\leq N$, $1\leq \ell\leq N^k$ and $1\leq j\leq N^{k+1}$. Using the same arguments in the proof of \cite[Theorem~3.3]{JiangRuan}, we have $\upM_{k+1}\leq \upM^*_k$ so that $\rho(\upM_{k+1})\leq\rho(\upM^*_k)$.

Now we prove that $\rho(\upM_k^*)=\rho(\upM_k)$. The proof is divided into two parts.  Firstly we show that $\rho(\upM_k)\geq \rho(\upM_k^*)$. Write $\gl=\rho(\upM_k^*)$. From Lemma~\ref{th:PFN}, $\gl$ is an eigenvalue of $\upM_k^*$ and there is a nonnegative nonzero vector $ u=(u_1,\ldots,u_{N^{k+1}})^T$ such that $\upM_k^* u =\gl u$. We define a vector $u'=(u_1',\ldots,u_{N^k}')^T$ by
\[
  u'_j=\sum_{p=(j-1)N+1}^{jN} u_p, \quad 1\leq j \leq N^k.
\]
It is clear that $u'$ is also nonnegative and nonzero. By using the same arguments in the proof of  \cite[Theorem~3.3]{JiangRuan}, we can obtain that $\upM_k u'=\gl u'$ so that $\gl$ is an eigenvalue of $\upM_k$. Hence,  $\rho(\upM_k^*)=\gl\leq \rho(\upM_k)$.

Secondly  we show that $\rho(\upM_k)\leq \rho(\upM_k^*)$. Without loss of generality, we may assume that $\mu:=\rho(\upM_k)>0$.  From Lemma~\ref{th:PFN}, $\mu$ is an eigenvalue of $\upM_k$ and there is a nonnegative nonzero vector $ v=(v_1,\ldots,v_{N^{k}})^T$ such that $\upM_k v =\mu v$. We define a vector $v'=(v_1',\ldots,v_{N^{k+1}}')^T$ by
\[
  v'_{(n-1)N^k + \ell} = \overline{s}_{n,\ell}^k v_\ell, \quad 1 \leq n \leq N, 1 \leq \ell \leq N^k.
\]
It is clear that $v'$ is nonnegative. Furthermore, it follows from $\upM_k v =\mu v$ that  for all $1 \leq n \leq N$ and $1 \leq j \leq N^{k-1}$,
\begin{equation}\label{eq:eigenv-Mk02}
	\mu v_{(n-1)N^{k-1}+j}
	= \sum_{t=(j-1)N+1}^{jN} \overline{s}_{n,t}^k v_t
	= \sum_{t=(j-1)N+1}^{jN} v'_{(n-1)N^k+t}.
\end{equation}
Thus $v'$ is a nonzero vector since otherwise, $v$ is a zero vector which is a contradiction.
For any $1 \leq n \leq N$ and $1 \leq \ell \leq N^k$, there exist $1 \leq n' \leq N$ and $1 \leq j' \leq N^{k-1}$ such that $\ell = (n'-1)N^{k-1}+ j'$. Thus,
\begin{align*}
(\upM_k^* v')_{(n-1)N^k+\ell}
= \overline{s}_{n,\ell}^k \sum_{p=(\ell-1)N+1}^{\ell N} v'_p  & = \overline{s}_{n,\ell}^k \sum_{t=(j'-1)N+1}^{j'N} v'_{(n'-1)N^k+t}    \\
&= \overline{s}_{n,\ell}^k \mu v_{(n'-1)N^{k-1}+j'}    \qquad \qquad (\textrm{By \eqref{eq:eigenv-Mk02}})\\
&= \mu \overline{s}_{n,\ell}^k v_\ell = \mu v'_{(n-1)N^k + \ell},
\end{align*}
which implies $\upM_k^* v'= \mu v'$ so that $\mu$ is an eigenvalue of $\upM_k^*$. Hence,  $\rho(\upM_k)=\mu\leq \rho(\upM_k^*)$.

From the above arguments, $\rho (\upM_{k+1}) \leq \rho (\upM_{k}^*)  = \rho (\upM_{k})$. Since $\rho(\upM_k)\geq 0$ for all $k$, we know that $\lim_{k\to\infty} \rho(\upM_k)$ exists. 
\end{proof}

We remark that compared to \cite[Theorem~3.3]{JiangRuan}, the different part in the above proof is the proof of $\rho(\upM_{k})\leq \rho(\upM_{k}^*)$.
Similarly, we can obtain the following result. 
\begin{theo}\label{thm:rho-lowstar-exist}
  For all $k\in \bZ^+$,
  \[
    \rho(\lowM_{k+1}) \geq  \rho (\lowM_k).
  \]
  As a result, $\lim_{k\to \infty} \rho(\lowM_k)$ exists, denoted by $\rho_*$.
\end{theo}

In the case that $\rho_*=\rho^*$, we denote the common value by $\rho_{\bf{S}}$. The following result has been proved in \cite[Proposition~3.5]{JiangRuan} under the assumption that $S_n$ is Lipschitz for all $n$. Our present proof only use the fact that $S_n$ is continuous for all $n$.

\begin{theo}\label{thm:rho-lowstar-equal}
 Assume that $|S_n|$ is positive on $I$ for all $1 \leq n \leq N$. Then $\rho_*=\rho^*$.
\end{theo}
\begin{proof}
	For any $1\leq n\leq N$, from the fact that $S_n$ is continuous and nonzero on $I$, we have 
	\[\underline{S}_n:=\min\{|S_n(x)|:\, x\in I\}>0.\]
	Fix $\varepsilon>0$. Since $S_n$ is uniformly continuous on $I$ for all $n$, we know that for sufficiently large $k$,
	\[
	  \overline{s}_{n,j}^k\leq  \underline{s}_{n,j}^k +\varepsilon \underline{S}_n \leq  (1 +\varepsilon)\underline{s}_{n,j}^k, \quad 1\leq n\leq N, \;\; 1\leq j\leq N^k,
	\]
	so that $\lowM_k \leq \upM_k \leq (1 +\varepsilon) \lowM_k$. Thus, 
	\begin{equation*}
	  \rho (\lowM_k) \leq \rho( \upM_k) \leq (1 +\varepsilon) \rho(\lowM_k)
	\end{equation*}
	for sufficiently large $k$.
	By letting $k$ tend to infinity, $\rho_* \leq \rho^* \leq (1 +\varepsilon) \rho_*$. 
	From the arbitrariness of $\varepsilon$, we have $\rho_*=\rho^*$. 
\end{proof}

\subsection{The irreducibility of vertical scaling matrices}

Recall that
\begin{equation*}
	\gamma(x)=\sum_{n=1}^{N} |S_n(x)|, \quad x\in I,
\end{equation*}
and $\gamma^*=\max_{x \in I} \gamma(x)$, $\gamma_*=\min_{x \in I} \gamma(x)$.
For any $k\in \bZ^+$, we define
\begin{equation*}
	\overline{\gamma}_k=\max_{1 \leq j \leq N^k}
	\sum_{n=1}^{N} \overline{s}_{n,j}^k,\quad
	\underline{\gamma}_k=\min_{1 \leq j \leq N^k}
	\sum_{n=1}^{N} \underline{s}_{n,j}^k.
\end{equation*}

Using the similar arguments in Theorem~\ref{thm:rho-lowstar-equal}, we can obtain that
\[
   \gamma^*=\lim_{k \to \infty} \overline{\gamma}_k, \quad \gamma_*=\lim_{k \to \infty} \underline{\gamma}_k. 
\]
  For every $k\in \bZ^+$, from \cite[Theorem~8.1.22]{HorJoh90},
$	
 \underline{\gamma}_k \leq \rho(\lowM_k) \leq \rho(\upM_k) \leq \overline{\gamma}_k.
$
Hence,
\begin{align*}
\gamma_* \leq \rho_* \leq \rho^*\leq\gamma^*.
\end{align*}
Thus, if $\gamma$ is a constant function on $I$, then $\gamma(x)=\rho_{\bf{S}}$ for all $x\in I$.

Using the same arguments in the proof of \cite[Lemma~3.2]{JiangRuan}, we have the following result.

\begin{lem}\label{lem:primitive}
  Assume that for each $1\leq n\leq N$, vertical scaling function $S_n$ is not identically zero on every subinterval of $I$. Then
  $(\upM_k)^k>0$ for all $k\in \bZ^+$. As a result, $\upM_k$ is primitive for all $k\in \bZ^+$.
\end{lem}

Similarly, we can show that $\lowM_k$ is primitive if $|S_n|$ is positive for each $1\leq n\leq N$. However, it is much more involved to prove the primitivity of $\lowM_k$ under general setting.  In this paper, we will show that $\lowM_k$ is primitive for sufficiently large $k$ if $\gamma_*\geq 1$ and $S_n$ has  finitely many zero points for each $1\leq n\leq N$.




Define the \emph{multiplicity of zero points} of $\bfS=\{S_n:1\leq n\leq N\}$ at $x \in I$ by 
\[\mathcal {ZM}(\bfS,x) =\card \{n: S_n(x)=0, 1\leq n \leq N \},\]
where $\card(A)$ is the cardinality of a set $A$.
Write $\mathcal {ZM}(\bfS) =\max_{x \in I} \mathcal {ZM}(\bfS,x) $. We have the following simple fact.

\begin{lem}\label{lem:gamma-1-p}
If $\gamma_* \geq 1$, then $\mathcal {ZM}( \bfS) \leq N-2$. 
\end{lem}
\begin{proof}
We prove this lemma by contradiction.
Assume that there exists $\wdt{x} \in I$ such that $\mathcal {ZM}(\bfS,\wdt{x}) \geq N-1$. Then there exists $1\leq n_0 \leq N$, such that $S_{n}(\wdt {x})=0$ for all $n \neq n_0$.
Hence,
\begin{align*}
	\gamma_* \leq  \gamma (\wdt x)= \sum_{n=1}^N |S_{n}(\wdt x)| = |S_{n_0} (\wdt x)| <1,
\end{align*}
which contradicts the fact that $\gamma_* \geq 1$.
\end{proof}

\begin{lem}\label{lem:zero point separation}
Assume that $\mathcal {ZM}(\bfS) \leq N-2$ and the function $S_n$ has finitely many zero points for all $1 \leq n \leq N$. Then there exists $k_1 \in \bZp$ such that for all $k>k_1$, every row of $\lowM_k$ has at least $N-1$ positive entries, and every column of $\lowM_k$ has at least $2$ positive entries.
\end{lem}
\begin{proof}

Let $Z_n$ be the set of zero points of $S_n$ for $1\leq n\leq N$ and write $Z=\bigcup_{n=1}^N Z_n$. 
Let $E$ be the set of endpoints of all $I_j^k$ for $1\leq j\leq N^k$ and $k\geq 1$, i.e.,  
\[
  E=\bigcup_{k\geq 1}\{x_0+jN^{-k}(x_N-x_0):\, 0\leq j\leq N^k\}.
\] 
Since $Z$ is a finite set, there exists a positive integer {$k_1$} satisfying the following two conditions: 
\begin{itemize}
  \item[(1)] $I_j^{k_1}$ contains at most one element of $Z$ for all $1\leq j\leq N^{k_1}$,
  \item[(2)] for every point $x\in Z\cap E$, there exists $1\leq j\leq N^{k_1}$ such that $x$ is the endpoint of $I_j^{k_1}$.
\end{itemize}
Then it is easy to see that for all $k>k_1$, every row of $\lowM_k$ has at least $N-1$ positive entries. 

Notice that $\mathcal {ZM}(\bfS,x)\leq N-2$ for all $x \in Z$.
Thus every column of $\lowM_k$ has at least $2$ positive entries. 
\end{proof}


\begin{lem}\label{lem:mk positive}
Under the assumptions of Lemma~\ref{lem:zero point separation}, for  all $k>k_1$, every row of $(\lowM_k)^k$ has at least $(N-1)^k$ positive entries and every column of $(\lowM_k)^k$ has at least $2^k$ positive entries. Here $k_1$ is the constant in Lemma~\ref{lem:zero point separation}.
\end{lem}
\begin{proof}


Fix $k>k_1$.
For all $m\geq 1$ and $1\leq i\leq N^k$, we define  
\[
   row_m(i) =\{j: \big((\lowM_k)^m\big)_{ij}>0\}.
\]

Notice that $\big((\lowM_k)^{m+1}\big)_{ij}=\sum_{t=1}^{N^k}(\lowM_k)_{it}\big( (\lowM_k)^m\big)_{tj}$ for all $m\geq 1$ and $1\leq i,j\leq N^k$. Thus for all $m\geq 1$ and $1\leq i\leq N^k$,
\[
  row_{m+1}(i)=\{j: \mbox{there exists } t\in row_1(i)\; \mbox{such that } j\in row_m(t) \}.
\]

It follows from the definition of $\lowM_k$ that for all $1\leq i\leq N$, $1\leq \ell\leq N^{k-1}$,
\begin{equation}\label{eq:row1}
  row_1((i-1)N^{k-1}+\ell) \subset \{(\ell-1)N+1,(\ell-1)N+2,\ldots,\ell N\}.
\end{equation}

We claim that for each $1\leq m\leq k-1$,
\[
   row_m((i-1)N^{k-m}+\ell) \subset \{(\ell-1)N^m+1,(\ell-1)N^m+2,\ldots, \ell N^m\}
\]
for all $1\leq i \leq N^m$ and $1\leq \ell\leq N^{k-m}$.

It follows from \eqref{eq:row1} that the claim holds for $m=1$. Assume that the claim holds for some $1\leq m\leq k-2$. Now given $1\leq i\leq N^{m+1}$ and $1\leq \ell\leq N^{k-(m+1)}$, we write $i'=(i-1)N^{k-(m+1)}+\ell$. If $j\in row_{m+1}(i')$, then there exists $t\in row_1(i')$ such that $j\in row_m(t)$. Notice that there exist unique integer pair $(i_1,i_2)$ with $1\leq i_1\leq N$ and $1\leq i_2\leq N^m$ such that $i=(i_1-1)N^m+i_2$. Thus
\[
  i'=(i_1-1)N^{k-1}+(i_2-1)N^{k-(m+1)}+\ell.
\]
Hence, from \eqref{eq:row1}, $(i_2-1) N^{k-m}+(\ell-1)N+1\leq t\leq (i_2-1)N^{k-m}+\ell N.$
Combining this with the inductive assumption, we have
$(\ell-1)N^{m+1}+1\leq j\leq \ell N^{m+1}$
so that the claim holds for $m+1$. This completes the proof of the claim.

It directly follows from the claim that for all $1\leq m\leq k-1$ and $1\leq i\leq N^k$, if $t_1\not=t_2\in row_1(i)$, then
$
  row_m(t_1)\cap row_m(t_2)=\emptyset,
$
which implies that
\begin{equation}\label{eq:row-recurrent}
  \card(row_{m+1}(i))=\sum_{t\in row_1(i)} \card(row_m(t)).
\end{equation}

From Lemma~\ref{lem:zero point separation}, $\card(row_1(i))\geq N-1$ for all $1\leq i\leq N^k$. Combining this with \eqref{eq:row-recurrent}, we can use inductive arguments to obtain that
$\card(row_m(i))\geq (N-1)^m$
for all $1\leq m\leq k$ and $1\leq i\leq N^k$. Thus every row of $(\lowM_k)^k$ has at least $(N-1)^k$ positive entries.

Similarly, for all $m\geq 1$ and $1\leq j\leq N^k$, we define
\[
  col_m(j) =\{i: \big( (\lowM_k)^m \big)_{ij}>0\}.
\]
Then for all $m\geq 1$ and $1\leq j\leq N^k$,
\[
  col_{m+1}(j)=\{i: \mbox{there exists } t\in col_1(j) \; \mbox{such that } i\in col_m(t) \}.
\]
By using similar arguments as above, we can obtain that for each $1\leq m\leq k-1$,
\[
  col_m((j-1)N^m+\ell)\subset \{j,j+N^{k-m},\ldots,j+(N^m-1)N^{k-m}\}
\]
for all $1\leq j\leq N^{k-m}$ and $1\leq \ell\leq N^m$. Hence, for all $1\leq m\leq k-1$ and $1\leq j\leq N^k$, if $t_1\not=t_2\in col_1(j)$, then
$col_m(t_1)\cap col_m(t_2)=\emptyset$.
As a result, we have $\card(col_m(j))\geq 2^m$ for all $1\leq m\leq k$ and $1\leq j\leq N^k$, which implies that every column of $(\lowM_k)^k$ has at least $2^k$ positive entries.
\end{proof}

The following result is part of the statement in \cite[8.5.P5]{HorJoh90}. We will use it to prove that $\lowM_k$ is primitive for sufficiently large $k$ under certain conditions.
\begin{lem}[\cite{HorJoh90}]\label{lem:prim-crit}
   Let $A=(a_{ij})_{n\times n}$ be an irreducible nonnegative matrix. Assume that at least one of its main diagonal entry $a_{ii} (1\leq i\leq n)$ is positive. Then $A$ is primitive.
\end{lem}

\begin{theo}\label{th:4-11}
Assume that $\mathcal {ZM}(\bfS) \leq N-2$ and the function $S_n$ has finitely many zero points for each $1 \leq n \leq N$. Then there exists $k_0 \in \bZp$ such that  $\lowM_k$ is primitive for all $k > k_0$.
\end{theo}
\begin{proof}


For every $1\leq n\leq N$ and $1\leq j\leq N^k$, we call $\underline{s}_{n,j}^k$ a \emph{basic entry} of the matrix $\lowM_k$. If all basic entries are positive, then using the same arguments in the proof of \cite[Lemma~3.2]{JiangRuan}, we can obtain that $(\lowM_k)^k>0$.  

In the case that $N=2$, we have $\mathcal {ZM}(\bfS) =0$ so that $\mathcal {ZM}(\bfS,x)=0$ for all $x\in I$. Thus $S_n(x)\not=0$ for $n=1,2$ and  all $x\in I$. It follows that all basic entries are positive so that $(\lowM_k)^k>0$. Hence $\lowM_k$ is primitive for all $k \in \mathbb{Z}^+$. 


Now we assume that $N\geq 3$.
Let $m_n$ be the number of zero points of $S_n$ on $I$. Write $m=\sum_{n=1}^{N}m_n$. Then for any $k\geq 1$, there are at most $2m$ basic entries equal to zero. Notice that for every $k\geq 1$ and $1\leq i,j\leq N^k$, the $(i,j)$ entry of $ (\lowM_k)^k$ is
\begin{equation}\label{eq:Thm4-11-1}
  \big((\lowM_k)^k\big)_{ij} = \sum_{1\leq t_2,\ldots,t_{k}\leq N^k\atop t_1=i,t_{k+1}=j} \prod_{\ell=1}^{k} (\lowM_k)_{t_\ell,t_{\ell+1}},
\end{equation}
and both every row and every column of $\lowM_k$ have $N$ basic entries. Hence, a zero basic entry of $\lowM_k$ can make at most $kN^{k-1}$ entries of $(\lowM_k)^k$ to be zero. Thus there are at most $2mkN^{k-1}$ zero entries in  $(\lowM_k)^k$.

Let $k_1$ be the constant in Lemma~\ref{lem:mk positive} and $k_0 =\max \{2,m, k_1\}$. We claim that $(\lowM_k)^k$ is irreducible for all $k>k_0$.

We prove the claim by contradiction. Assume  that $(\lowM_k)^k$ is reducible. Then there are nonempty and disjoint subsets $A,B$ of $\{1, \ldots, N^k\}$ satisfying $A \cup B =\{1, \ldots, N^k\}$, and for all $i\in A$ and $j \in B$, the $(i,j)$ entry of $(\lowM_k)^k$ is zero. From Lemma~\ref{lem:mk positive} and $N-1\geq 2$, there are at least $2^k$ elements in both $A$ and $B$. Hence 
\[
  \card(A)\cdot \card(B)=\card(A)\cdot (N^k-\card(B))\geq 2^k(N^k-2^k)
\]
so that $(\lowM_k)^k$ has at least $2^k(N^k-2^k)$ zero entries. From $N\geq 3$ and $k > k_0$, we have $2^k>k^2>mk$ and $N^k-2^k>2N^{k-1}$ so that $2^k(N^k-2^k) > 2mk N^{k-1}$, which is a contradiction. This completes the proof of the claim.

Now we will show that for all $k>k_0$, at least one of the main diagonal entry of $(\lowM_k)^k$ is positive, so that $(\lowM_k)^k$ is primitive by Lemma~\ref{lem:prim-crit}. As a result, $\lowM_k$ is primitive for all $k>k_0$.
  
For any $j \in \{1,2,\ldots,N^k\}$, there exists a unique finite sequence $j_1, \ldots, j_k \in \{1,\ldots,N\}$ such that 
\begin{equation*}
	j=(j_1-1)N^{k-1}+(j_2-1)N^{k-2}+\cdots+(j_{k-1}-1)N+j_k.
\end{equation*}
We define $\sigma(j)=(j-(j_1-1)N^{k-1})N+j_1$. Then $\sigma(j)\in\{1,2,\ldots,N^k\}$. Thus we can define $\sigma^p(j)=\sigma(\sigma^{p-1}(j))$ for $p\geq 2$. It is easy to see that $\sigma^k(j)=j$ and $(\lowM_k)_{j,\sigma(j)}$ is a basic entry of $\lowM_k$ for all $j\in \{1,\ldots,N^k\}$.

 
Write $\sigma^0(j)=j$.    
From \eqref{eq:Thm4-11-1}, $\big((\lowM_k)^k\big)_{jj}\geq\prod_{p=1}^{k} (\lowM_k)_{\sigma^{p-1}(j),\sigma^p(j)}$.
Hence, a zero basic entry of $\lowM_k$ can make at most $k$ main digonal entries of $(\lowM_k)^k$ to be zero. Thus there are at most $2mk$ zero main digonal entries in  $(\lowM_k)^k$.

Notice that $k_0\geq \max\{2,m\}$ and $N\geq 3$. Hence, for $k> k_0$, we have $N^k\geq 3^k>2k^2>2mk$ so that $(\lowM_k)^k$ contains at least one positive main diagonal entry.
\end{proof}

From Lemma~\ref{lem:gamma-1-p} and Theorems~\ref{thm:rho-upstar-exist}, \ref{thm:rho-lowstar-exist}, \ref{thm:rho-lowstar-equal} and \ref{th:4-11}, we know that Theorem~\ref{thm:Main-Results-1} holds.


\section{Proof of Theorem~\ref{thm:Main-Results-2}}
In the rest of the paper, we always assume that $f$ is a generalized affine FIFs satisfying conditions (A1)-(A3). 

\subsection{Box dimension estimate of the graph of continuous functions}
Given a bounded subset $E$ of $\mathbb{R}^d$, we use $\overline{\dim}_{B} E$ and $\underline{\dim}_{B} E$ to denote the upper box dimension and the lower box dimension of $E$, respectively. If $\overline{\dim}_{B} E=\underline{\dim}_{B} E$, then we use  $\dim_B E$ to denote the common value and call it the \emph{box dimension} of $E$.
It is well known that $\underline{\dim}_{B}E\geq 1$ when $E$ is the graph of a continuous function on a closed interval of $\mathbb{R}$. Please see ~\cite{Fal14} for details.

Let $g$ be a continuous function on  $J$. For any $U\subset J$, we use $O(g,U)$ to denote the oscillation of $g$ on $U$, that is,
\begin{equation*}
	O(g,U)= \sup\limits_{x',x'' \in U}|g({x}')-g({x}'')|.
\end{equation*}
Write
\begin{equation*}
	O_k(g,J)=\sum\limits_{j=1}^{N^k} O(g,J_j^k),
\end{equation*}
where $J_j^k$ is defined by \eqref{def:J-jk}. 

The following lemma presents a method to estimate the upper and lower box dimensions of the graph of a function by its oscillation. Similar results can be found in \cite{Fal14,KRZ18,RSY09}.

\begin{lem}[\cite{JiangRuan}]\label{lem:box-dim-new}
	Let $g$ be a continuous function on a closed interval $J$. Then
	\begin{align*}
		&\underline{\dim}_B \Gamma g \geq 1+\varliminf_{k\to\infty}\frac{\log \big( O_k(g,J)+1\big)}{k\log N}, 	\quad \mbox{and} \\
		&\overline{\dim}_B \Gamma g \leq 1+\varlimsup_{k\to\infty} \frac{\log \big(O_k(g,J)+1\big)}{k\log N}.
	\end{align*}
\end{lem}

We remark that $J=[0,1]$ in the original version of the above lemma in \cite{JiangRuan}. However, it is straightforward to see that the lemma still holds in the present version.  

It is clear that $\{O_k(g,J)\}_{k=1}^\infty$ is increasing with respect to $k$. Thus $\lim_{k\to \infty} O_k(g,J)$ always exists.
Write $\Var(g,J)$ the classical total variation of $g$ on $J$. We have the following simple fact.
\begin{lem}\label{lem:VarOinfty}
	Let $g$ be a continuous function on a closed interval $J=[a,b]$. Then $\lim_{k\to \infty} O_k(g,J)=\Var(g,J)$.
\end{lem}
\begin{proof}
	Clearly, $O_k(g,J)\leq \Var(g,J)$ for all $k\in \bZp$. Thus $\lim_{k\to \infty} O_k(g,J)\leq \Var(g,J)$. Now we prove the another inequality. 
	
	Arbitrarily pick a partition $T=\{a=t_0<t_1<\cdots<t_n=b\}$ of $J$. Fix $k\in \bZp$ large enough such that $N^{-k}<\min\{t_{i}-t_{i-1}:\, 1\leq i\leq n\}$. For every $0\leq i\leq n$, there exists $\ga_i\in\{1,\ldots,N^k\}$ such that $t_i\in J_{\ga_i}^k$. Furthermore, it is easy to see that $1=\ga_0<\ga_1<\cdots<\ga_n=N^k$. Notice that for any $1\leq i\leq n$,
	\[
	|g(t_i)-g(t_{i-1})|\leq \sum_{p=\ga_{i-1}}^{\ga_i} O(g,J_p^k).
	\]
	Thus
	\begin{align*}
		\sum_{i=1}^{n}|g(t_i)-g(t_{i-1})| &\leq \sum_{i=1}^{n}\sum_{p=\ga_{i-1}}^{\ga_i} O(g,J_p^k)\\
		&=  O_k(g,J)+\sum_{i=1}^{n-1} O(g,J_{\ga_i}^k)\leq \lim_{k\to \infty}O_k(g,J)+\sum_{i=1}^{n-1} O(g,J_{\ga_i}^k).
	\end{align*}
	Since $g$ is continuous on $I$, we can choose $k$ large enough such that $\sum_{i=1}^{n-1} O(g,J_{\ga_i}^k)$ as small as possible.  Hence
	\[
	\sum_{i=1}^{n}|g(t_i)-g(t_{i-1})| \leq \lim_{k\to \infty}O_k(g,J).
	\]
	By the arbitrariness of the partition $T$, $\Var(g,J)\leq \lim_{k\to \infty}O_k(g,J)$.
\end{proof}

\subsection{Estimate of oscillations}

By the definition of $W_n$, it is easy to see that $W_n(x,y)=(L_n(x),F_n(x,y)$, where
\[
  L_n(x)=(x-x_0)/N+x_{n-1}, \quad F_n(x,y)=S_n(x)y+q_n(x).
\]
From  \eqref{eq:1-5}, $W_n(x,f(x))=(L_n(x),f(L_n(x)))$. Thus, we have the following useful equality:
\begin{equation}\label{eq:FIF-Rec-Expression}
	f(L_n(x))=S_n(x)f(x)+q_n(x), \quad x\in [x_0,x_N], \; n=1,2,\ldots,N.
\end{equation}

Write $M_f=\max_{x \in I}|f(x)|$.
By using the similar arguments in the proof of \cite[Lemma~4.2]{JiangRuan}, we can obtain the following lemma. 


\begin{lem}\label{lem:osci-1}
	For any $1\leq n\leq N$ and $D \subset I$,
		\begin{align*}
			O(f,L_n(D))	&\leq  \sup_{x \in D}\big|S_n(x)\big| O(f,D)+	M_f O(S_n,D)	+ O(q_n,D) , \quad \mbox{and}\\
			O(f,L_n(D))& \geq  \inf_{x \in D}\big|S_n(x)\big| O(f,D)-	M_f O(S_n,D)	- O(q_n,D) .
		\end{align*}
\end{lem}

\begin{proof}
		From \eqref{eq:FIF-Rec-Expression},
		\begin{align*}
			O(f,L_n(D))
			=&\sup_{x',x'' \in D}\big|S_n(x')f(x')-S_n(x'')f(x'') + q_n(x')-q_n(x'')\big|   \\
			\leq & \sup_{x',x'' \in D}\big|S_n(x') \big( f(x')-f(x'')\big) \big| +
			\sup_{x',x'' \in D}\big|f(x'') \big(S_n(x')-S_n(x'')\big) \big| \\
			&+ \sup_{x',x'' \in D}\big| q_n(x')-q_n(x'')\big| \\
			\leq & \sup_{x \in D}\big|S_n(x)\big| O(f,D)+	M_f O(S_n,D)	+ O(q_n,D) .
		\end{align*}
		On the other hand, we choose $x',x'' \in  D$ such that $O(f,D)=\big|f(x')-f(x'')\big|$. Then
	\begin{align*}
		&O(f,L_n(D)) \\
		\geq & |f(L_n(x'))-f(L_n(x''))|\\
		\geq & \big| S_n(x')\big(f(x')-f(x'')\big)   \big| -\big|f(x'') \big(S_n(x')-S_n(x'')\big) \big| -\big| q_n(x')-q_n(x'')\big| \\
		\geq & \inf_{x \in D}\big|S_n(x)\big| O(f,D)-	M_f O(S_n,D)	- O(q_n,D).
	\end{align*}
	Thus the lemma holds.
\end{proof}

Using the argument similar to the proof of the first part of the above lemma, we have the following result.
\begin{lem}\label{lem:osci-2}
	For any $1\leq n\leq N$, $D \subset I$ and $t \in D$,
	\begin{align*}
		\big|O(f,L_n(D))	- |S_n(t)| O(f,D)   \big|\leq  2	M_f O(S_n,D)	+ O(q_n,D)  .
	\end{align*}
\end{lem}
\begin{proof}
	For any $x',x''\in D$, we have 
	\[
	  \big|S_n(x')f(x')-S_n(x'')f(x'') - S_n(t)(f(x')-f(x''))\big| \leq 2M_fO(S_n,D). 
	\]
	Thus
	\begin{align*}
		|S_n(t)|O(f,D)-2M_f O(S_n,D) 
		\leq  &\big|S_n(x')f(x')-S_n(x'')f(x'')| \\
		\leq &|S_n(t)|O(f,D)+2M_f O(S_n,D)
	\end{align*}
	so that the lemma holds.
\end{proof}

Given $k,p \in \mathbb{Z}^+$ and $g \in C(I)$, we define
\begin{equation*}
	V(g,k,p)=\big(O_p(g,I_1^k),O_p(g,I_2^k),\ldots,O_p(g,I_{N^k}^k)\big)^T \in \mathbb{R}^{N^k},
\end{equation*}
and call it an \emph{oscillation vector} of $g$ with respect to $(k,p)$.
It is obvious that
\begin{equation*}
	O_{k+p}(g,I)={\lVert{V}(g,k,p)\rVert}_1,
\end{equation*}
where $\lVert v\rVert_1:=\sum_{i=1}^d |v_i|$ for any $v=(v_1,\ldots,v_d)\in \bR^d$.



Define a vectors $\xi_k$ in $\bR^{N^k}$ by 
\begin{equation}\label{eq:u-beta-bfq-k-def}
	(\xi_{k})_{(n-1)N^{k-1}+\ell}= M_f \Var(S_n, I_{\ell}^{k-1})+\Var(q_n, I_{\ell}^{k-1}),  
\end{equation}
where $1\leq n\leq N, 1\leq \ell \leq N^{k-1}$.

\begin{lem}\label{lem:2-3}
	For any $k \in \mathbb{Z}^+$ and any $p \in \bZp$, 
	\begin{equation}\label{eq:4-3}
		-\xi_{k}+\lowM_k V(f,k,p)\leq V(f,k,p+1)\leq \xi_{k}+\upM_k V(f,k,p).
	\end{equation}
\end{lem}
\begin{proof}
	From Lemma~\ref{lem:osci-1}, for any $1\leq n\leq N$, $k\in \bZp$, $1\leq j\leq N^k$ and any $D\subset I_j^k$,
	\[
	O(f,L_n(D))\leq \overline{s}_{n,j}^k O(f,D) + O(q_n, D) + M_f O(S_n,D) .
	\]
	Notice that $\big(L_n(I_j^k)\big)_m^p = L_n\big((I_j^k)_m^p\big)$ for $1\leq m\leq N^p$. Thus,
	\begin{align*}
		O_p(f,L_n(I_{j}^k))&=\sum_{m=1}^{N^p} O\Big(f,\big(L_n(I_{j}^k)\big)_m^p\Big) \\
		&\leq \sum_{m=1}^{N^p}\Big( \overline{s}_{n,j}^k O(f,(I_{j}^k)_m^p)) +O(q_n, (I_j^k)_{m}^p) +M_f O(S_n,(I_{j}^k)_m^p)\Big)\\
		&=\overline{s}_{n,j}^k O_p(f,I_j^k)+ O_p(q_n,I_j^k)+ M_f O_p(S_n, I_j^k) .
	\end{align*}
	Hence, from
	$I_{\ell}^{k-1}=\bigcup_{j=(\ell-1)N+1}^{\ell N} I_{j}^k,$

	\begin{align*}
		O_{p+1}(f,L_n(I_{\ell}^{k-1}))
		&=\sum_{j=(\ell-1)N+1}^{\ell N} O_p(f,L_n(I_{j}^k)) \\
		&\leq \sum_{j=(\ell-1)N+1}^{\ell N} \Big( \overline{s}_{n,j}^{k}	 O_p(f,I_{j}^{k}) +O_p(q_n,I_{j}^{k})+M_f O_p(S_n,I_{j}^{k}) \Big) \\
		&\leq M_f O_{p+1}(S_n,I_{\ell}^{k-1})+O_{p+1}(q_n,I_{\ell}^{k-1})+\sum_{j=(\ell-1)N+1}^{\ell N} \overline{s}_{n,j}^k O_p(f,I_j^k).
	\end{align*}
	By the definitions of $\xi_k$ and $\upM_k$, we can rewrite this inequality as
	\[
	  O_{p+1}(f,I_{(n-1)N^{k-1}+\ell}^k) \leq (\xi_k)_{(n-1)N^{k-1}+\ell}+\big(\upM_{k} V(f,k,p)\big)_{(n-1)N^{k-1}+\ell}
	\]
	so that  $V(f,k,p+1)\leq \xi_{k}+\upM_k V(f,k,p)$.  Similarly, we can prove that another inequality in \eqref{eq:4-3} holds.
\end{proof}


\subsection{Estimate the box dimension of  $\Gamma f$ by $\rho_*$ and $\rho^*$}

\begin{theo}\label{th:upper-Bdim}
Assume that the function $S_n$ is not identically zero on every subinterval of $I$ for all $1 \leq n \leq N$.
Then
\begin{equation}\label{eq:upBoxFormula}
	\overline{\dim}_{B} \Gamma f  \leq  \max\big\{1,1+\log_N  \rho^*\big\}.
\end{equation}

\end{theo}
\begin{proof}
Fix $k\in \bZp$.
Let $\xi_{k}$ be the vector in $\bR^{N^k}$ defined by \eqref{eq:u-beta-bfq-k-def}. From Lemma~\ref{lem:primitive},  $\upM_k$ is primitive so that it is irreducible.
By Lemma~\ref{th:PF}, we can choose a positive eigenvector $w_k$ of $\upM_k$ such that $w_k\geq \xi_k$ and $w_k\geq V(f,p,1)$. Hence, from Theorem~\ref{lem:2-3}, we have
\begin{align*}
	V(f,k,p+1)\leq w_k+\upM_k V(f,k,p)
\end{align*}
for all $p\in \bZp$. Thus,
\begin{align*}\label{eq:upper-key-con}
	V(f,k,p) &\leq w_k+ \upM_k w_k+\cdots+(\upM_k)^{p-2}w_k+(\upM_k)^{p-1} V(f,k,1) \\
             &\leq \sum_{\ell=0}^{p-1}\rho(\upM_k)^{\ell} w_k
\end{align*}
for all $p\in \bZ^+$. It follows that
\begin{equation*}
  O_{k+p}(f,I)=||V(f,k,p)||_1\leq ||w_k||_1 \sum_{\ell=0}^{p-1} (\rho(\upM_k))^{\ell}\leq ||w_k||_1 p \big(\rho(\upM_k)^p+1\big).
\end{equation*}
Hence,
\begin{align*}
   \varlimsup_{p\to\infty} \frac{\log (O_{k+p}(f,I)+1)}{p\log N} 
    \leq \max \Big\{ 0, \frac{\log \rho (\upM_k)}{\log N}  \Big\}.
\end{align*}
Thus, from Lemma~\ref{lem:box-dim-new},
\begin{align*}
    \overline{\dim}_B \Gamma f \leq 1+\varlimsup_{p\to\infty} \frac{\log (O_{k+p}(f,I)+1)}{p\log N} 
    \leq \max \Big\{ 1, 1+\frac{\log \rho (\upM_k)}{\log N}  \Big\}.
\end{align*}

By the arbitrariness of $k$, we know from Theorem~\ref{thm:rho-upstar-exist} that \eqref{eq:upBoxFormula} holds.
\end{proof}

\begin{theo}\label{th:lower-Bdim}

Assume that $\Var(f,I)=\infty$, $\mathcal {ZM}(\bfS) \leq N-2$ and the function $S_n$ has finitely many zero points on $I$ for all $1 \leq n \leq N$. Then
\begin{align}\label{eq:thm-lowBdim}
	\lowBdim \Gamma f	\geq 1+\log_N \rho_*.
\end{align}

\end{theo}

\begin{proof}
Notice that $\lowBdim \Gamma f	\geq 1$ always holds. Thus, without loss of the generality, we may assume that $\rho_*> 1$. From 
$\rho_*=\lim_{k\to \infty}\rho(\lowM_k)$, there exists a positivie integer $k_2$, such that $\rho(\lowM_k)>1$ for all $k>k_2$.
Let $k_0$ be the constant in Theorem~\ref{th:4-11}. From Theorem~~\ref{th:4-11}, $\lowM_k$ is primitive so that it is irreducible for all $k>k_0$.

Fix $k >\max\{k_0,k_2\}$. 
Given $1<\tau<\rho(\lowM_k)$, from Lemma~\ref{th:PF}, we can find a positive eigenvector $w_k$ of $\lowM_k$ with eigenvalue $\rho(\lowM_k)$ such that $w_k \geq \xi_k/(\rho(\lowM_k)-\tau)$.
Since $\lowM_k$ is primitive, there exists $\ell_k\in \bZp$ such that $\big(\lowM_k\big)^{\ell_k}>0$. Let $\alpha_k$ be the minimal entry of the matrix $\big(\lowM_k\big)^{\ell_k}$. Then $\alpha_k>0$.
From Theorem~\ref{lem:2-3},
\begin{equation}\label{eq:lower-connect}
	V(f,k,p+1)\geq \lowM_k V(f,k,p) -\xi_k
\end{equation}
for all $p\in \bZp$. Repeatedly using this inequality, we can obtain that for all $p\in \bZp$,
\begin{equation}\label{eq:5-3}
	V(f,k,p+\ell_k)\geq (\underline{M}_k)^{\ell_k} V(f,k,p) -\sum_{q=0}^{\ell_k-1} (\underline{M}_k)^{q} \xi_k.
\end{equation}
Notice that the maximal entry of $V(f,k,p)$ is at least $N^{-k}\lVert V(f,k,p)\rVert_1$. Thus,
\begin{equation*}
  (\underline{M}_k)^{\ell_k} V(f,k,p) \geq (\alpha_k',\ldots,\alpha_k'),
\end{equation*}
where $\alpha_k'=\alpha_kN^{-k}\lVert V(f,k,p)\rVert_1$. Notice that
\[
  \lim_{p\to \infty} \lVert V(f,k,p)\rVert_1=\lim_{p\to \infty} O_{k+p}(f,I)=\Var(f,I)=\infty.
\]
Hence, we can choose $p_*$ large enough such that
\[
  (\underline{M}_k)^{\ell_k} V(f,k,p_*) \geq w_k+\sum_{q=0}^{\ell_k-1} (\underline{M}_k)^{q} \xi_k.
\]
Let $p_k=p_*+\ell_k$. Then from \eqref{eq:5-3},
\[
  V(f,k,p_k)\geq w_k\geq  \frac{1}{\rho(\lowM_k)-\tau} \xi_k.
\]

From \eqref{eq:lower-connect},
\[
  V(f,k,p_k+1)\geq\rho(\lowM_k)w_k-\xi_k\geq \rho(\lowM_k)w_k-(\rho(\lowM_k)-\tau)w_k=\tau w_k.
\]
Notice that for all $\ell\in \bZ^+$,
\begin{align*}
  \rho(\lowM_k)\tau^\ell w_k-\xi_k &= \rho(\lowM_k)\big(\tau^\ell-1\big)w_k+\rho(\lowM_k)w_k-\xi_k \\
  & \geq \tau\big(\tau^{\ell}-1\big)w_k+\tau w_k=\tau^{\ell+1}w_k.
\end{align*}
Thus, by induction, $V(f,k,p_k+\ell) \geq \tau^{\ell} w_k$ for all $\ell\in \bZ^+$.
Hence
\[
	O_{k+p_k+\ell}(f,I) = \lVert V(f,k,p_k+\ell)\rVert_1  \geq \tau^{\ell} \lVert w_k\rVert_1,
\]
which implies that
\begin{equation*}
	\varliminf_{\ell \to \infty}{\frac{\log \big(O_{\ell}(f,I)+1\big)}{\ell \log N}}
   =\varliminf_{\ell \to \infty}{\frac{\log \big(O_{k+p_k+\ell}(f,I)+1\big)}{\ell \log N}}
	\geq \frac{\log \tau}{\log N} .
\end{equation*}
It follows from the arbitrariness of $\tau$ that $\log  \rho(\lowM_k) /\log N$ is less than the left hand side of this inequality.
Combining this with Lemma~\ref{lem:box-dim-new}, we have 
\[
  \lowBdim \Gamma f\geq 1+\frac{\log  \rho(\lowM_k)}{\log N}.
\] 
Since this result holds for all $k>\max\{k_0,k_2\}$, we know from Theorem~\ref{thm:rho-lowstar-exist} that
\eqref{eq:thm-lowBdim} holds.
\end{proof}

\begin{rem}
  From the proof of the above theorem, it is easy to see that under the assumptions of the theorem, $\Var(f,I_j^k)=\infty$ for any $k\in\bZp$ and $1\leq j\leq N^k$.
\end{rem}

\begin{theo}\label{thm-bdim}
	Under the assumption of Theorem~\ref{th:lower-Bdim} and the additional assumption that $\rho_*=\rho^*$, if $\Var(f,I)=\infty$ and $\rho_{\bfS}>1$, then
\begin{equation}\label{eq:thm-Bdim}
	 \Bdim \Gamma f=1+\log_N  \rho_{\bfS},
\end{equation}
otherwise $\Bdim  \Gamma f=1$.
\end{theo}
\begin{proof}
In the case that $\Var(f,I)<\infty$, we know from Lemma~\ref{lem:box-dim-new} that $\upBdim \Gamma f\leq 1$.
In the case that $\rho_{\bfS}\leq 1$, we know from Theorem~\ref{th:upper-Bdim} that $\upBdim \Gamma f\leq 1$. Since $\lowBdim \Gamma f \geq 1$ always holds, $\Bdim  \Gamma f=1$ if $\Var(f,I)<\infty$ or $\rho_{\bfS}\leq 1$.

In the case that $\Var(f,I)=\infty$ and $\rho_{\bfS}>1$, we know from Theorems~\ref{th:upper-Bdim} and \ref{th:lower-Bdim} that \eqref{eq:thm-Bdim} holds.
\end{proof}

From Lemma~\ref{lem:gamma-1-p} and Theorems~\ref{th:upper-Bdim}, \ref{th:lower-Bdim} and \ref{thm-bdim}, we know that Theorem~\ref{thm:Main-Results-2} holds.
Furthermore, from Theorems~\ref{thm:rho-lowstar-equal} and \ref{thm-bdim}, we have the following result.

\begin{cor}\label{cor:5-7}
Assume that the function $|S_n|$ is positive on $I$ for each $1 \leq n \leq N$.
Then in the case that $\Var(f,I)=\infty$ and $\rho_{\bfS}>1$,
\eqref{eq:thm-Bdim} holds,
otherwise $\Bdim  \Gamma f=1$.
\end{cor}

\section{Estimate the box dimension of FIFs by $\gamma^*$ and $\gamma_*$}

In this section, we will estimate the box dimension of $\Gamma f$ by the sum function of vertical scaling functions.
By Lemma~\ref{lem:osci-2}, we can obtain the following result.



\begin{lem}\label{lem:Osci-rec-est}
	For all $k\in \bZp$,
	\begin{align}
		& O_{k+1}(f,I)\leq \gamma^* \cdot O_k(f,I)+\sum_{n=1}^{N} \big(\Var(q_n,I)+2M_f \Var(S_n,I)\big), \quad {and} \label{eq:Osci-simpRec-1} \\
		& O_{k+1}(f,I)\geq \gamma_* \cdot  O_k(f,I)-\sum_{n=1}^{N} \big(\Var(q_n,I)+2M_f \Var(S_n,I)\big). \label{eq:Osci-simpRec-2}
	\end{align}
\end{lem}
\begin{proof}
	Given $D\subset I$, we know from Lemma~\ref{lem:osci-2} that for any $t\in D$,
	\begin{align*}
		\sum_{n=1}^N O(f, L_n(D)) &\leq \gamma(t) \cdot O(f,D) + \sum_{n=1}^N \big(O(q_n,D)+2 M_f O(S_n,D)\big)  \\
		&\leq \gamma^* \cdot O(f,D) + \sum_{n=1}^N \big(O(q_n,D)+2M_f O(S_n,D)\big).
	\end{align*}
	For any $k\in \bZp$ and $1\leq j\leq N^k$, by letting $D=I_j^k$ in the above inequality, we have
	\[
	\sum_{n=1}^N O(f,L_n(I_j^k)) \leq \gamma^* \cdot O(f,I_j^k) + \sum_{n=1}^{N} \big( O(q_n,I_j^k)+2 M_f O(S_n,I_j^k)\big).
	\]
	Hence   	
		\begin{align*}
			O_{k+1} (f,I)  &= \sum_{n=1}^{N} \sum_{j=1}^{N^k} O(f,L_n(I_j^k)) \\ 
			& \leq \gamma^* \cdot O_k(f,I)+\sum_{n=1}^{N} \big( O_k(q_n,I)+ 2M_f O_k(S_n,I)\big)\\
			& \leq \gamma^* \cdot  O_k(f,I)+\sum_{n=1}^{N} \big( \Var(q_n,I)+2 M_f  \Var(S_n,I)\big) ,
	\end{align*} 
	so that \eqref{eq:Osci-simpRec-1} holds. Similarly, we can prove that \eqref{eq:Osci-simpRec-2} holds. 
\end{proof}

From this lemma, we can obtain the upper box dimension estimate by $\gamma^*$ and the lower box dimension estimate by $\gamma_*$.

\begin{theo}\label{thm:Bdim-gammaStar}
	We have $\upBdim \Gamma f \leq \max\{1, 1+\log_N \gamma^*\}$. Furthermore, if $\gamma_*>1$ and $\Var(f,I)=\infty$, then $\lowBdim \Gamma f \geq 1+\log_N \gamma_*$.
\end{theo}
\begin{proof}
	Write $\eta=\sum_{n=1}^{N} (\Var(q_n,I)+2M_f\Var(S_n,I))$. It is clear that $\eta<\infty$ since $S_n$ and $q_n$ are of bounded variation on $I$ for each $n$.   If $\gamma^*\leq 1$,  from Lemma~\ref{lem:Osci-rec-est},
	\[
	O_{k+1}(f,I)\leq O_k(f,I)+\eta, \quad \forall k\geq 1,
	\] 
	so that
	\[
	O_k(f,I)\leq O_1(f,I)+(k-1)\eta, \quad \forall k\geq 1.
	\]
	Thus from Lemma~\ref{lem:box-dim-new}, $\upBdim \Gamma f\leq 1=\max\{1, 1+\log_N \gamma^*\}$.
	
	In the case that $\gamma^*> 1$, we know from Lemma~\ref{lem:Osci-rec-est} that
	\[
	O_{k+1}(f,I)+\frac{\eta}{\gamma^*-1}\leq \gamma^*\Big(O_k(f,I)+\frac{\eta}{\gamma^*-1}\Big), \quad \forall k\geq 1,
	\] 
	so that
	\[
	O_k(f,I)+\frac{\eta}{\gamma^*-1}\leq (\gamma^*)^{k-1}\Big(O_1(f,I)+\frac{\eta}{\gamma^*-1}\Big), \quad \forall k\geq 1.
	\]
	Thus from Lemma~\ref{lem:box-dim-new}, $\upBdim \Gamma f\leq 1+\log_N \gamma^*=\max\{1, 1+\log_N \gamma^*\}$. 
	
	Now we assume that $\gamma_*>1$ and $\Var(f,I)=\infty$. Using Lemma~\ref{lem:Osci-rec-est} again, we have
	\begin{equation}\label{eq:gammaLowstar-temp1}
		O_{k+1}(f,I)-\frac{\eta}{\gamma_*-1}\geq \gamma_*\Big(O_k(f,I)-\frac{\eta}{\gamma_*-1}\Big), \quad \forall k\geq 1.
	\end{equation}
	Since  $\Var(f,I)=\infty$, from Lemma~\ref{lem:VarOinfty}, there exists $k_0\in \bZp$ such that $O_{k_0}(f,I)>\eta/(\gamma_*-1)$. From \eqref{eq:gammaLowstar-temp1}, 
	\[
	O_k(f,I)-\frac{\eta}{\gamma_*-1}\geq (\gamma_*)^{k-k_0}\Big(O_{k_0}(f,I)-\frac{\eta}{\gamma_*-1}\Big), \quad \forall k\geq k_0.
	\]
	Thus from Lemma~\ref{lem:box-dim-new}, $\lowBdim \Gamma f\geq 1+\log_N \gamma_*$. 
	Hence, the theorem holds.
\end{proof}

\begin{rem}\label{remark:3-5}
	From the proof of Theorem~\ref{thm:Bdim-gammaStar}, it is easy to see that under the condition $\gamma_*>1$, the following two properties are equivalent:
	\begin{enumerate}
		\item $\Var(f,I)=\infty$,
		\item there exists $k_0\in \bZp$ such that 
		\[
		O_{k_0}(f,I)>(\gamma_*-1)^{-1} \sum_{n=1}^{N} (\Var(q_n,I)+2M_f\Var(S_n,I)).
		\]
	\end{enumerate}
\end{rem}

\begin{rem}\label{remark:3-6}
	Under the condition that the function $S_n$ is nonnegative for each $n$,  from  \eqref{eq:FIF-Rec-Expression},
	\[
	\sum_{n=1}^N f(L_n(x))
	=\sum_{n=1}^N  \big(S_n(x) f(x)+q_n(x)\big)		
	=\gamma(x)f(x)+\sum_{n=1}^N  q_n(x).		
	\] 
	Thus, by using arguments similar to the proof of \cite[Theorem 4.10]{JiangRuan},  we have
	\[
    	O_{k+1}(f,I)\geq \gamma_*O_k(f,I)- {M_f \Var(\gamma,I)} -\Var\Big(\sum_{n=1}^N q_n,I\Big).
	\]
	Thus, if $\gamma_*>1$ and the function $S_n$ is nonnegative on $I$ for each $1 \leq n \leq N$, then $\Var(f,I)=\infty$ if and only if there exists $k_0 \in \mathbb{Z}^+$ satisfying 
	\begin{equation*}
		O_{k_0}(f,I)> (\gamma_*-1)^{-1} \Big( M_f \Var(\gamma,I) +\Var(\sum_{n=1}^{N}q_n,I)\Big).
	\end{equation*}
\end{rem}

From Theorem~\ref{thm:Bdim-gammaStar}, we can obtain the following result.

\begin{theo}\label{thm:Bdim-gamma}
	Assume that $\gamma(x)\equiv\gamma_0$ for all $x\in I$. Then in the case that $\gamma_0>1$ and $\Var(f,I)=\infty$,
	\begin{equation}\label{eq:thm-Bdim-gamma}
		\Bdim \Gamma f=1+\log_N \gamma_0,
	\end{equation}
	otherwise $\Bdim \Gamma f=1$. 
\end{theo}
\begin{proof}
	Notice that  $\lowBdim \Gamma f\geq 1$ always holds since $f$ is a continuous function on $I$.
	In the case that $\gamma_0\leq 1$, it follows from Theorem~\ref{thm:Bdim-gammaStar} that $\upBdim \Gamma f\leq 1$. In the case that $\Var(f,I)<\infty$, we have $\lim_{k\to \infty} O_k(f,I)<\infty$. Thus, from  Lemma~\ref{lem:box-dim-new}, $\upBdim \Gamma f\leq 1$. Hence $\Bdim \Gamma f =1$ if $\gamma_0\leq 1$ or $\Var(f,I)<\infty$.
	
	Now we assume that $\gamma_0>1$ and $\Var(f,I)=\infty$. From Theorem~\ref{thm:Bdim-gammaStar},
	\[\upBdim \Gamma f\leq 1+\log_N \gamma_0\leq \lowBdim \Gamma f.\]
	so that  $\Bdim \Gamma f=1+\log_N \gamma_0$. Thus \eqref{eq:thm-Bdim-gamma} holds.
\end{proof}


    From Remark~\ref{remark:3-6} and Theorem~\ref{thm:Bdim-gamma}, we have the following result.
	\begin{cor}\label{cor:3-8}
		Assume that the function $S_n$ is nonnegative for each $n$, and  both $\gamma$ and $\sum_{n=1}^{N}q_n$ are constant functions on $I$. Then in the case that $\gamma(0)>1$ and $f$ is not a constant function, $\dim_B \Gamma f=1+\log_N \gamma(0)$, otherwise $\dim_B \Gamma f=1$.
	\end{cor}

\section{An example and further remarks}

\subsection{An example: generalized Weierstrass-type functions}

Weierstrass functions are classical fractal functions. There are many works on fractal dimensions of their graphs, including the box and Hausdorff dimension. Please see \cite{HuLau93,KMY84,RenShen21} and the references therein. For example, Ren and Shen \cite{RenShen21} studied the following Weierstrass-type functions
\begin{equation*}
	g_{\lambda,N}^\phi(x)=\sum_{k=0}^\infty \lambda^k \phi (N^k x), \quad x \in \mathbb{R},
\end{equation*}
where $N\geq 2$ is an integer, $1/N  < \lambda<1$ and $\phi: \mathbb{R} \to \mathbb{R}$ is a $\mathbb{Z}$-periodic real analytic function. They proved that either such a function is real analytic, or the Hausdorff dimension of its graph is equal to $2+\log_N\lambda$.

It is well known that $f=g_{\lambda,N}^\phi \big|_{[0,1]}$ is a generalized affine FIF. In fact,
 for $n \in \{1, 2,\ldots, N\}$ and $x \in [0,1]$, we have
\begin{align*}
	f\Big(\frac{x+n-1}{N}\Big) 
	= \phi \Big( \frac{ x+n-1}{N}\Big)+ \lambda \sum_{k=0}^\infty \lambda^k \phi (N^{k}x) 
	= \phi \Big( \frac{ x+n-1}{N}\Big)+ \lambda f(x)   .
\end{align*}
Thus, $\Gamma f=\bigcup_{n=1}^N W_n(\Gamma f)$, where for $n=1,2,\ldots,N$,
\begin{align*}
	W_n(x,y) =\Big(\frac{x+n-1}{N},\lambda y+  \phi\big(\frac{x+n-1}{N}\big)\Big), \quad (x,y)\in [0,1]\times \mathbb{R}.
\end{align*}

Let $\phi(x)=\cos(2\pi x)$. Then $g_{\lambda,N}^\phi$ is the classical Weierstrass function. Shen \cite{Shen18} proved that the Hausdorff dimension of its graph is equal to $2+\log_N \lambda$. Let $q_n(x)=\cos(2\pi(x+n-1)/N)$, $1\leq n\leq N$. It is easy to check that $\sum_{n=1}^{N} q_n(x)=0$ for all $x\in [0,1]$. Thus, from Corollary~\ref{cor:3-8}, we obtain the well known result $\dim_B \Gamma f=2+\log_N \lambda$, where $f=g_{\lambda,N}^\phi\big|_{[0,1]}$ and  $\phi(x)=\cos(2\pi x)$.

By Theorem~\ref{thm:Main-Results-2}, we can study the box dimension of generalized Weierstrass-type functions by replacing vertical scaling factor $\lambda$ with vertical scaling functions.

\begin{figure}[htbp]
    \centering
    \includegraphics[width=5cm]{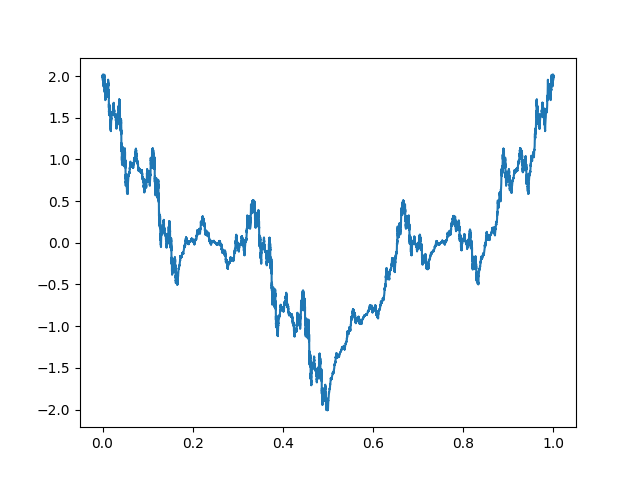}
    \caption{The FIF in Example~\ref{exam-1}}
    \label{fig:FIF-Exam-1}
\end{figure}

\begin{exam}\label{exam-1}

Let $I=[0,1]$, $N=3$, and $x_n=n/3$, $n=0,1,2,3$. Let vertical scaling functions $S_n$, $1\leq n\leq 3$ on $[0,1]$ are defined by
\begin{align*}
	S_1(x)=S_2(x)=\frac{1}{2}+\frac{\sin (2 \pi x)}{4}, \qquad
	S_3(x) = \frac{1}{2}-\frac{\sin (2 \pi x)}{4}.
\end{align*}
Then each function $S_n$ is positive on $I$ so that $\rho_*=\rho^*$.

Let $\phi(x)=\cos (2\pi x)$ and define maps $W_n$, $1 \leq n \leq 3$ by
\[
  W_n(x,y)=\Big(\frac{x+n-1}{3},S_n(x)y+\phi\big(\frac{x+n-1}{3}\big)\Big), \quad (x,y)\in [0,1]\times \bR.
\]

Let $x_n=n/3$ for $0\leq n\leq 3$. Let $y_0=y_2=2$ and $y_1=y_3=1/2$. Then it is easy to check that 
\[
  W_n(x_0,y_0)=(x_{n-1},y_{n-1}), \quad W_n(x_3,y_3)=(x_n,y_n)
\]
for $n=1,2,3$. Thus $\{W_n\}_{n=1}^3$ determines a generalized affine FIF $f$. Please see Figure~\ref{fig:FIF-Exam-1} for the graph of $f$.

Notice that $\gamma(x)=\sum_{n=1}^{3}|S_n(x)|=3/2+\sin(2\pi x)/4$  for $x\in [0,1]$.
Hence, $\gamma_* = 5/4$, $\gamma^*=7/4$ and $\lambda'= \pi/2 $ is a Lipschitz constant of $\gamma(x)$.

Let $q_n(x)=\phi((x+n-1)/3)$, $x\in [0,1]$, $n=1,2,3$. Then $\sum_{n=1}^{3}q_n(x)=0$ for all $x\in [0,1]$ so that 
$\Var(\sum_{n=1}^{3} q_n,I)=0$.

Now we calculate $M_f=\max\{|f(x)|:\, x\in I\}$. Notice that for any $x\in I$, there exists $n_1n_2\cdots \in \{1,2,3\}^\infty$ such that $x\in \bigcap_{k=1}^\infty L_{n_1}\circ L_{n_2}\circ\cdots \circ L_{n_k}(I)$. Thus from \eqref{eq:FIF-Rec-Expression}, we have
\begin{align*}
  f(x)&=q_{n_1}(L_{n_1}^{-1}(x))+S_{n_1}(L_{n_1}^{-1}(x))f(L_{n_1}^{-1}(x)) \\
      &=q_{n_1}(L_{n_1}^{-1}(x))+\sum_{k=2}^{\infty} \Big( \prod_{t=1}^{k-1} S_{n_t}\big(L_{n_t}^{-1} \circ \cdots \circ L_{n_1}^{-1}(x)\big) \Big) q_{n_k}\big( L_{n_k}^{-1}\circ \cdots \circ L_{n_1}^{-1}(x)\big).
\end{align*}
Hence, from $q^*:=\max\{|q_n(x)|: x\in [0,1], n=1,2,3\}=1$ and
\[
  S^*:=\max\{S_n(x): x\in [0,1], n=1,2,3\}=\frac{3}{4},
\] 
we have
$M_f\leq q^*\sum_{k=0}^{\infty}(S^*)^k= q^*/(1-S^*)=4$.
 Thus,
\begin{equation*}
	\frac{\lambda' M_f |I| + \Var(\sum_{n=1}^{3} q_n,I)}{\gamma_*-1} 
	\leq \frac{(\pi /2) \times 4 \times 1 +0}{5/4-1}= 8 \pi.
\end{equation*}
By calculation, $O_{6}(f,I)>8 \pi$. 
Thus, from Remark ~\ref{remark:3-6}, $\Var(f,I)=\infty$.

By definition of vertical scaling matrices, we have
\begin{align*}
\upM_1=
\begin{pmatrix}
	\frac{3}{4} & \frac{1}{2}+ \frac{\sqrt{3}}{8} & \frac{1}{2} \\
	\frac{3}{4} & \frac{1}{2}+ \frac{\sqrt{3}}{8} & \frac{1}{2} \\	
	\frac{1}{2} & \frac{1}{2}+ \frac{\sqrt{3}}{8} & \frac{3}{4} \\
\end{pmatrix},
\quad
\underline{M}_1=
\begin{pmatrix}
	\frac{1}{2} & \frac{1}{2}- \frac{\sqrt{3}}{8} & \frac{1}{4} \\
	\frac{1}{2} & \frac{1}{2}- \frac{\sqrt{3}}{8} & \frac{1}{4} \\
	\frac{1}{4} & \frac{1}{2}- \frac{\sqrt{3}}{8} & \frac{1}{2} \\
\end{pmatrix}.
\end{align*}
In general, by calculation, we can obtain the spectral radii of vertical scaling matrices $\rho(\upM_k)$ and $\rho(\lowM_k)$, $k=1,2,4,5,7,8$ as in Tabel 1. Thus, from Theorem~\ref{thm:Main-Results-2},
\[
	\dim_B \Gamma f =1 + \log_N \rho_{\bfS} \approx 1 +\log1.516/\log3 \approx 1.379.
\]

\begin{table}[htbp]
	\centering
	\begin{tabular}{|c|c|c|c|c|c|c|c|c|}
	\hline 												
	$k$ &  1	&  2 	&  4 & 5  & 7 &8   \\	\hline
	$\rho\big(\upM_k\big)$ 	&1.95688&1.68984 &1.53627&1.52277 &1.51675 &1.51625		\\	
     \hline	$\rho\big(\lowM_k\big)$ &1.05567 &1.33590 &1.49577&1.50926 &1.51525&1.51575\\	
     \hline    
	\end{tabular}
	\caption{$\rho(\upM_k)$ and $\rho(\lowM_k)$ in Example~\ref{exam-1}}
	\label{table:rho}
\end{table}

\end{exam}

\subsection{Further remarks}

From the proof of Theorem~\ref{thm:rho-upstar-exist}, we essentially prove that $\rho^*=\lim_{k\to \infty} \rho(\upM_k)$ exists without any restrictions on vertical scaling functions. This also holds for the existence of $\rho_*=\lim_{k\to \infty} \rho(\lowM_k)$. Hence, from Theorem~\ref{thm:Main-Results-2}, we have the following conjecture.
\begin{conj}
    Let $f$ be a generalized affine FIF satisfying conditions (A1)-(A3). Then $\rho_*=\rho^*$. Furthermore, in the case that $\mathrm{Var}(f,I)=\infty$
	and $\rho_{\bfS}>1$, $\Bdim \Gamma f=1+\log_N  \rho_{\bfS}$,	otherwise $\Bdim  \Gamma f=1$.
\end{conj}

\medskip

\bibliographystyle{amsplain}

\end{document}